\documentclass[a4,12pt]{article}%
\usepackage{a4wide}
\usepackage{graphicx}
\usepackage{graphics}
\usepackage{xcolor}
\usepackage{setspace,mathrsfs}
\newcommand{\Eqref}[1]{Eq.\,\eqref{#1}}
\usepackage[maxbibnames=9]{biblatex}
\bibliography{stochastic}

\usepackage{amsmath}
\usepackage{amsfonts}
\usepackage{amssymb}
\usepackage{enumerate}
\newtheorem{theorem}{Theorem}[section]

\newtheorem{corollary}[theorem]{Corollary}

\newtheorem{definition}[theorem]{Definition}

\newtheorem{lemma}[theorem]{Lemma}

\newtheorem{remark}[theorem]{Remark}

\newenvironment{proof}[1][Proof]{\noindent\textbf{#1.} }{\ \rule{0.5em}{0.5em}}

\newcommand{\E}{{\rm \bf E}}
\renewcommand{\P}{{\rm \bf P}}

\newcommand{\prob}{{\rm \bf P}}

\newcommand{\dN}{{\mathbb N}}
\newcommand{\N}{{\mathbb N}}

\newcommand{\dR}{{\mathbb R}}

\newcommand{\ep}{\varepsilon}

\newcommand{\black}{\color{black}}

\newcounter{figurecounter}
\begin{document}

\title{Repeated Games with Tail-Measurable Payoffs%
\thanks{
Solan acknowledges the support of the Israel Science Foundation, grant
\#217/17.}}

\author{J\'{a}nos Flesch\footnote{Department of Quantitative Economics, 
Maastricht University, P.O.Box 616, 6200 MD, The Netherlands. E-mail: j.flesch@maastrichtuniversity.nl.}
\and Eilon Solan\footnote{School of Mathematical Sciences, Tel-Aviv University, Tel-Aviv, Israel, 6997800, E-mail: eilons@tauex.tau.ac.il.}}

\maketitle

\begin{abstract}
We study multiplayer Blackwell games, which are repeated games where the payoff of each player is a bounded and Borel-measurable function of the infinite stream of actions played by the players during the game. These games are an extension of the two-player perfect-information games studied by David Gale and Frank Stewart (1953). Recently, various new ideas have been discovered to study Blackwell games. In this paper, we give an overview of these ideas by proving, in four different ways, that Blackwell games with a finite number of players, finite action sets, and tail-measurable payoffs admit an $\ep$-equilibrium, for all $\ep>0$.
\end{abstract}

\textbf{Keywords:} Repeated game, equilibrium, tail-measurable payoff.

\section{Introduction}\label{sec-intro}

In a seminal paper, David Gale and Frank Stewart (1953) studied two-player win-lose games with perfect information and infinite horizon, where player~1 wins if the realized play lies in a given set of plays,
and player 2 wins otherwise.
They proved that if the winning set of player~1 is open or closed in the product topology,
then the game is determined: one of the players has a winning strategy.
Along the years this result has been extended to larger classes of winning sets,
until Martin (1975) proved that the game is determined as soon as the winning set is Borel-measurable.
When the winning set is not Borel-measurable, the game is not always determined.

Blackwell (1969) studied analogous games where the two players move simultaneously:
in every stage both players simultaneously choose actions,
and player~1 wins if and only if the play is in some given winning set of plays.
These games are termed \emph{Blackwell games},
and in the language of game theory they are repeated games with general payoff functions.\footnote{The qualification ``general'' refers to the fact that the payoff is not necessarily the discounted sum of stage payoffs or the long-run average of stage payoffs.}
Blackwell (1969) proved that if the action sets are finite and the winning set is $G_\delta$ in the product topology, that is,
a countable intersection of open sets, then the game has a value.
The restriction to finite action sets is natural, since even in one-shot games
with countable action sets, the value need not exist.
Along the years, the value existence result has been extended to larger classes of winning sets,
until Martin (1998) extended it to all Borel-measurable sets, 
by reducing the existence problem to the determinacy of two-player win-lose alternating-move games.
In fact, Martin's (1998) result does not require that the outcome is determined by a winning set: as soon as the outcome is a bounded and Borel-measurable function of the play, the value exists, and as a consequence, each player has an $\ep$-optimal strategy for any error-term $\ep>0$. 

As the determinacy of two-player zero-sum games has been established,
it is natural to go beyond zero-sum games 
and ask whether an $\ep$-equilibrium exists in multiplayer Blackwell games with bounded and Borel-measurable payoff functions, for every $\ep > 0$. 
This is a challenging and largely uncharted question,%
\footnote{Note that multiplayer Blackwell games include stochastic games with deterministic transitions: in these games, the play is in one of the states in each stage, and the next state is determined, in a deterministic way, by the action combination chosen in the current state.} and only some special classes have been shown to have an $\ep$-equilibrium. These include multiplayer Blackwell games with
perfect information (Mertens and Neyman, see Mertens, 1987),%
\footnote{See also Le Roux and Pauly (2014). For the concept of subgame perfect $\ep$-equilibrium in perfect information games, we refer to Flesch and Predtetchinski (2017), Bruy\`ere, Le Roux, Pauly, Raskin~(2021), and the references therein.} with discounted payoffs (see, e.g., the survey by Ja\'{s}kiewicz and Nowak (2018)), with two-players and the long-run average payoff (Vieille (2000a,b)), and games with classical computer science objectives (see, e.g., Secchi and Sudderth (2002), Chatterjee, Majumdar, Jurdzinski (2004), Chatterjee (2005), and Bouyer, Brenguier, Markey, and  Ummels (2015)).  

Recently, various new ideas have been discovered to study multiplayer Blackwell games when the payoff functions of the players are tail-measurable. The condition of tail-\hyphenate{measurability} amounts to requiring that the payoffs are independent of the actions chosen in any finite number of stages. In this paper, we give an overview of these ideas, and we do so by proving, in four different ways, that multiplayer Blackwell games with finite action sets, and bounded, Borel-measurable and tail-measurable payoffs admit an $\ep$-equilibrium, for all $\ep>0$. 
The first proof applies to Blackwell games with two players, and 
the other proofs
apply to Blackwell games with finitely many players.
The third proof 
can even be extended 
to Blackwell games with countably many players. While the first proof is standard, and the second and 
third 
have been published, see Ashkenazi-Golan, Flesch, Predtetchinski, and Solan (2022a, 2022b), the 
fourth 
proof is new.

We hope that the variety of approaches taken to prove the result will attract researchers to study the colorful and fascinating area of Blackwell games,
and help advance the research in this area.

\section{Model and Main Result}\label{sec2}

Throughout the paper,
the set of natural numbers is denoted by $\N = \{1,2,\ldots\}$.

\subsection{The Model}
\label{section:model}

\begin{definition}\label{def-game}
A \emph{Blackwell game} is a triplet
$\Gamma = (I,(A_i)_{i \in I}, (f_i)_{i \in I})$, where
\begin{itemize}
\item $I$ is a nonempty and finite set of \emph{players}.
\item   $A_i$ is a nonempty and finite \emph{action set} for player~$i$, for each $i \in I$. Let $A = \prod_{i \in I} A_i$ denote the set of action profiles, and $A^\dN$ denote the set of \emph{plays}, i.e., infinite streams of action profiles. We endow the set $A^\dN$ of plays with the product topology $\mathcal{F}(A^\dN)$, where the action set $A_i$ of each player $i\in I$ has its natural discrete topology. 
\item   $f_i : A^\dN \to \dR$ is a bounded and Borel-measurable \emph{payoff function} for player~$i$, for each $i \in I$.
\end{itemize}
\end{definition}

The game is played in discrete time as follows. In each stage $t \in \dN$,
each player $i \in I$ selects an action $a_i^t \in A_i$, simultaneously with the other players. This induces an action profile $a^t=(a_i^t)_{i\in I}$ in stage $t$, which is observed by all players. Given the realized play $p = (a^1,a^2,\ldots)$, the payoff to each player~$i \in I$ is $f_i(p)$. \medskip

\noindent\textbf{Histories.} A \emph{history} is a finite sequence of elements of $A$, including the empty sequence $\emptyset$. 
The set of histories is denoted by $H = \bigcup_{t=0}^\infty A^t$,
where $A^t$ is the $t$-fold product of $A$. For a play $p \in A^\dN$ write $p^{\leq t}$ to denote the prefix of $p$ of length $t$. 

For two histories $h,h'\in H$ and a play $p\in A^\dN$, we denote the concatenation of $h$ and $h'$ by $hh'$ and that of $h$ and $p$ by $hp$.

Each history induces a subgame of $\Gamma$. Given $h\in H$, the \emph{subgame} that starts at $h$ is the game $\Gamma^h = (I,(A_i)_{i \in I}, (f_{i,h})_{i \in I})$, where $f_{i,h} = f_i \circ s_h$ and $s_h : A^{\dN} \to A^{\dN}$ is given by $p \mapsto hp$.\medskip

\noindent\textbf{Strategies.}
A \emph{mixed action} for player $i$ is a probability measure $x_i$ on $A_i$.
The set of mixed actions for player $i$ is denoted by $X_i$.
A \emph{mixed action profile} is a vector of mixed actions $x = (x_i)_{i \in I}$, one for each player.
The set of mixed action profiles is denoted by $X$.

A (behavior) \emph{strategy} of player~$i$ is a function $\sigma_i : H \to X_i$. 
The interpretation is that when player~$i$ adopts the strategy $\sigma_i$,
at each
history $h\in H$ player~$i$ chooses 
an action according to the probability measure $\sigma_i(h)$. We denote by $\Sigma_i$ the set of strategies of player~$i$. 

A \emph{strategy profile} is a vector of strategies $\sigma = (\sigma_i)_{i \in I}$, one for each player. We denote by 
$\Sigma = \prod_{i \in I} \Sigma_i$ 
the set of strategy profiles.\medskip

\noindent\textbf{Definitions for the opponents of a player.}
For $i\in I$, let $-i$ denote the set of player $i$'s opponents, $I \setminus \{i\}$. 
When restricting attention to players in $-i$, 
we can similarly define mixed action profiles $x_{-i} = (x_j)_{j\neq i}$ and strategy profiles $\sigma_{-i} = (\sigma_j)_{j\neq i}$. The corresponding sets are denoted by $X_{-i}$ and $\Sigma_{-i}$, respectively.\medskip

\noindent\textbf{Expected payoffs.} Using Kolmogorov's extension theorem, each strategy profile $\sigma$ induces a unique probability measure $\mathbb{P}_\sigma$ on $(A^\dN,\mathcal{F} (A^\dN))$. Player $i$'s expected payoff under $\sigma$ is 
\[\E_{\sigma}[f_i]\,=\,\int_{A^\dN}f_i(p)\ \mathbb{P}_\sigma(dp),\]
and similarly, player $i$'s expected payoff under $\sigma$ in the subgame starting at a history $h$ is
\[\E_{\sigma}[f_i \mid h]\,=\,\int_{A^\dN}f_{i,h}(p)\ \mathbb{P}_{\sigma_h}(dp),\]
where $\sigma_h$ is the strategy profile defined as $\sigma_h(h')=\sigma(hh')$ for each history $h'$.\medskip

\noindent\textbf{Equilibrium.} Let $\ep\geq 0$. A strategy profile $\sigma^*$ is called an \emph{$\ep$-equilibrium}, if for every player $i \in I$ and every strategy $\sigma_i \in \Sigma_i$,
we have $\E_{\sigma^*}[f_i] \,\geq\, \E_{\sigma_i,\sigma^*_{-i}}[f_i] - \ep$. This means that no player can gain more than $\ep$ by a unilateral deviation. \medskip

\noindent\textbf{Minmax value and maxmin value.} The minmax value  (also called the upper-value) is the highest payoff that a player can defend against any strategy profile of her opponents. Formally, player $i$'s \emph{minmax value} in the game is defined as
\begin{equation*}
\overline{v}_i\, =\, \inf_{\sigma_{-i} \in \Sigma_{-i}}\,\sup_{\sigma_i \in \Sigma_i}\E_{\sigma_i,\sigma_{-i}}[f_i],
\end{equation*}
and similarly, player $i$'s minmax value in the subgame starting at 
a history $h$ is
\begin{equation*}
\overline{v}_i(h)\, = \,\inf_{\sigma_{-i} \in \Sigma_{-i}}\,\sup_{\sigma_i \in \Sigma_i}\E_{\sigma_i,\sigma_{-i}}[f_i\mid h].
\end{equation*}
Note that if $\sigma^*$ is an $\ep$-equilibrium, then each player $i$'s payoff is at least her minmax value minus $\ep$, i.e., $\E_{\sigma^*}[f_i]\geq \overline{v}_i-\ep$.

The maxmix value (also called the lower-value) is the highest payoff that a player can guarantee to receive. Formally, player $i$'s \emph{maxmin value} in the game is defined as
\begin{equation*}
\underline{v}_i\, =\, \sup_{\sigma_i \in \Sigma_i}\,\inf_{\sigma_{-i} \in \Sigma_{-i}}\E_{\sigma_i,\sigma_{-i}}[f_i],
\end{equation*}
and similarly, player $i$'s maxmin value in the subgame starting at 
a history $h$ is
\begin{equation*}
\underline{v}_i(h)\, = \,\sup_{\sigma_i \in \Sigma_i}\,\inf_{\sigma_{-i} \in \Sigma_{-i}}\E_{\sigma_i,\sigma_{-i}}[f_i\mid h].
\end{equation*}

It follows by the definitions that $\overline{v}_i\geq \underline{v}_i$ and $\overline{v}_i(h)\geq \underline{v}_i(h)$ for each $h\in H$, and it is known that these inequalities can be strict.
In the special case of two-player Blackwell games, by Martin (1998), the minmax and the maxmin values are equal: for each player $i\in I=\{1,2\}$ we have $\overline{v}_i= \underline{v}_i$ and $\overline{v}_i(h)= \underline{v}_i(h)$ for each $h\in H$. In particular, if the two-player Blackwell game is zero-sum, i.e., $f_1=-f_2$, then $\overline{v}_1=\underline{v}_1=-\overline{v}_2=-\underline{v}_2$, and this common quantity is called the \emph{value} of the game.
\medskip
\black

\noindent\textbf{Tail-measurability.} A function $g:A^\dN\to\dR$ is called \emph{tail-measurable}, if whenever two plays $p = (a^1,a^2,\dots)$
and $p' = (a'^1,a'^2,\dots)$ satisfy $a^t = a'^t$ for every $t$ sufficiently large, then $g(p)=g(p')$. In other words, $g$ is tail-measurable if changing finitely many coordinates in the play does not influence the value of $g$. Not every Borel-measurable function is tail-measurable, and not every tail-measurable function is Borel-measurable, see Rosenthal (1975) and Blackwell and Diaconis (1996). 
Various important evaluation functions in the literature are tail-measurable, for example, the limsup and the liminf of the average stage payoffs (e.g., Sorin (1992)), the limsup and the liminf payoffs (e.g., Maitra and Sudderth (1993)). Various classical winning conditions in the computer science literature, such as the Büchi, co-Büchi, parity, Streett, and Müller  (e.g., Gimbert and Horn (2008), and Chatterjee and Henzinger (2012)) are also tail-measurable.
The discounted payoff (e.g., Shapley (1953)) is not tail-measurable.

One important consequence of tail-measurability is that the minmax value 
and the maxmin value are independent of the history. 

\begin{lemma}
\label{indepvalue}
Consider a Blackwell game $\Gamma$ and a player $i\in I$. If player $i$'s payoff function $f_i$ is  tail-measurable, then $\overline{v}_i=\overline{v}_i(h)$ 
and $\underline{v}_i=\underline{v}_i(h)$ 
hold for every history $h\in H$.
\end{lemma}

\subsection{Main Theorem}

We now present the main theorem of the paper.

\begin{theorem}
\label{theorem:main}
Let $\Gamma = (I,(A_{i})_{i \in I},(f_{i})_{i \in I})$ be a Blackwell game, where the set $I$ of players is finite, the action set $A_i$ of each player $i\in I$ is finite, and the payoff function $f_i$ of each player $i\in I$ is bounded, Borel-measurable, and tail-measurable. Then, the Blackwell game $\Gamma$ admits an $\ep$-equilibrium, for every $\ep > 0$.
\end{theorem}

In the following sections, we will provide four different proofs for Theorem \ref{theorem:main}. The first proof is standard and it applies to Blackwell games with two players, the second and 
fourth 
proofs apply to Blackwell games with finitely many players, while the
third proof 
even applies to 
Blackwell games with 
countably many players. 


Fix a Blackwell game $\Gamma$ that satisfies the conditions of Theorem \ref{theorem:main}, and fix $\ep>0$ for the rest of the paper. 

The following lemma presents a useful sufficient condition for the existence of an $\ep$-equilibrium. This condition is well known, but for completeness we provide a proof.

\begin{lemma}
\label{suff-cond}
Assume that there is a play $p^*\in A^\dN$ such that $f_i(p^*)> \overline{v}_i-\ep$ for each player $i\in I$. Then, $\Gamma$ admits an $\ep$-equilibrium.
\end{lemma}

\begin{proof}
Choose $\delta > 0$ so that for each player $i\in I$,
\begin{equation}
\label{delta-def}
\delta\,\leq\, f_i(p^*)-(\overline{v}_i-\ep).
\end{equation}
By Lemma~\ref{indepvalue},
$\overline{v}_i=\overline{v}_i(h)$
for each history $h\in H$ and each player $i\in I$. 
Thus,  by the definition of the minmax value, 
for each history $h\in H$ and each player $i\in I$, 
there exists a strategy profile $\sigma^{h,i}_{-i}$ for player $i$'s opponents such that for each strategy $\sigma_i$ of player $i$ we have
\begin{equation}
\label{pun}
\E_{\sigma_i,\sigma^{h,i}_{-i}}[f_i\mid h]\,\leq\,\overline{v}_i+\delta.
\end{equation}
Intuitively, the strategy profile $\sigma^{h,i}_{-i}$ can be seen as a punishment
against player $i$.

Denote $p^*=(a^{*1}, a^{*2},\ldots)$. For each player $i\in I$, define a strategy $\sigma^*_i$ as follows. For any history $h$ at any stage $t\in\dN$:
\begin{itemize}
\item
If, along $h$, each player played consistently with $p^*$, that is, $h=(a^{*1},\ldots,a^{*t-1})$, then $\sigma^*_i(h):=a^{*t}_i$.
\item
Otherwise, let $t'$ be the first stage at which a player, say player $j$, deviated from $p^*$ along $h$.\footnote{If 
more than one
player deviated from $p^*$ at 
stage $t'$, then $j$ is 
the deviator with the smallest index.} Let $h'=(a^{*1},\ldots,a^{*t'})$,
and
set $\sigma^*_i(h) := \sigma^{h',j}_i(h)$ for each $i\neq j$, while $\sigma^*_j(h)$ is arbitrary.
\end{itemize}
Intuitively, the players are supposed to follow the play $p^*$. 
If, however, a player deviates from $p^*$, 
all other players
switch to a strategy profile that 
ensures
that the deviator's expected payoff is at most her minmax value plus $\delta$.

It is straightforward to verify that the strategy profile $\sigma^*=(\sigma^*_i)_{i\in I}$ is an $\ep$-equilibrium.
Indeed, $\sigma^*$ induces the play $p^*$, and hence under $\sigma^*$ each player $i$'s expected payoff satisfies 
\begin{equation}
\label{eq111}
\E_{\sigma^*}[f_i] \,=\, f_i(p^*) \,\geq\, \overline{v}_i-\ep+\delta.
\end{equation} 
Consider now an arbitrary strategy
$\sigma_i$ for some player $i\in I$. Let $\mu$ be the probability that the realized play under $(\sigma_i,\sigma^*_{-i})$ is $p^*$:
\[\mu:=\P_{\sigma_i,\sigma^*_{-i}}(\{p^*\}).\]
Due to the punishment strategies (cf. \Eqref{pun}), we have
\begin{equation}
\label{eq222}
\E_{\sigma_i,\sigma^*_{-i}}[f_i] \,\leq\, \mu\cdot f_i(p^*)+(1-\mu)\cdot (\overline{v}_i+\delta).
\end{equation} 
By Eqs.~\eqref{delta-def}, \eqref{eq111}, and \eqref{eq222}, 
\[\E_{\sigma_i,\sigma^*_{-i}}[f_i] -\E_{\sigma^*}[f_i] \,\leq\,(1-\mu)\cdot (\overline{v}_i+\delta-f_i(p^*))\,\leq\,\ep,\]
which completes the proof.
\end{proof}\\


One implication of
Lemma \ref{suff-cond} is the following: 
if there exists a strategy profile that, in each subgame, gives each player a payoff of at least her minmax value minus some $\delta\in[0,\ep)$, then there exists an $\ep$-equilibrium.  

\begin{corollary}
\label{suff-corr}
Assume that there exist a strategy profile $\sigma$ and some $\delta\in[0,\ep)$ such that $\E_\sigma[f_i\mid h]\geq \overline{v}_i-\delta$ for each player $i\in I$ and each history $h\in H$. Then, $\Gamma$ admits an $\ep$-equilibrium.
\end{corollary}

\begin{proof}
By L\'{e}vy's zero-one law, we have $\P_\sigma(f_i\geq \overline{v}_i-\delta)=1$ for each player $i\in I$. Since $I$ is finite, there is a play $p^*$ with $f_i(p^*)\geq \overline{v}_i-\delta>\overline{v}_i-\ep$ for each player $i\in I$. Hence, Lemma \ref{suff-cond} implies that $\Gamma$ admits an $\ep$-equilibrium.
\end{proof}

\section{First Proof: Subgame-Perfect Maxmin Strategies}
\label{sec-firstproof}

In this section, we prove Theorem~\ref{theorem:main} when there are only two players: $I= \{1,2\}$. The proof uses strategies, termed  \emph{subgame-perfect $\delta$-maxmin strategies}, that ensure that this player's expected payoff is at least her maxmin value minus $\delta$ in all subgames. 

Because there are only two players, as mentioned earlier, the minmax and the maxmin values are equal for each player: $\overline{v}_i=\underline{v}_i$ for each player $i\in I$. 

\begin{definition}
\label{def:maxmin}
Consider a Blackwell game. Let $i \in I$ and $\delta > 0$.
A strategy $\sigma_i^\delta$ for player $i$ is \emph{subgame-perfect $\delta$-maxmin}
if for every history $h\in H$ and every strategy profile $\sigma_{-i}\in\Sigma_{-i}$ of player $i$'s opponents,
\begin{equation*} 
\E_{\sigma_i^\delta,\sigma_{-i}}[f_i\mid h] \,\geq\, \underline{v}_i(h) - \delta.
\end{equation*}
\end{definition}

The following theorem is shown in Mashiah-Yaakovi (2015); see also Flesch, Herings, Maes, Predtetchinski (2021) and Ashkenazi-Golan, Flesch, Predtetchinski, Solan (2022a).

\begin{theorem}[Mashiah-Yaakovi (2015)]
\label{Ayala}
Consider a Blackwell game.
For every $i \in I$ and $\delta > 0$,
player $i$ has a subgame-perfect $\delta$-maxmin strategy.
\end{theorem}

We can now turn to the proof of Theorem~\ref{theorem:main} when $|I|=2$.\medskip

\begin{proof}[Proof of Theorem~\ref{theorem:main} when $|I|=2$]
For each $i \in I = \{1,2\}$,
let $\sigma_i$ be a subgame-perfect $\frac{\ep}{2}$-maxmin strategy of player~$i$. Then, 
by Lemma \ref{indepvalue}, 
for each player $i \in \{1,2\}$ and each history $h\in H$
\[
\E_{\sigma_1,\sigma_2}[f_i\mid h] \,\geq\,  \underline{v}_i-\tfrac{\ep}{2}.\]
 As $\underline{v}_i=\overline{v}_i$ for each player $i\in \{1,2\}$, by Corollary \ref{suff-corr}, the game $\Gamma$ has an $\ep$-equilibrium.
\end{proof}

\begin{remark}\label{rem-stoch}\rm 
By using subgame-perfect $\delta$-maxmin strategies, Flesch and Solan (2022) proved a result that is strongly related to Theorem~\ref{theorem:main} when there are only two players: each two-player stochastic game with finite state and action spaces and bounded, Borel-measurable, and shift-invariant payoffs admits an $\ep$-equilibrium, for all $\ep>0$. This result allows to have finitely many states, but uses shift-invariance of the payoffs, which is 
a stronger property
than tail-measurability.
\end{remark}

\begin{remark}\rm 
The reader may wonder whether the proof works in the presence of more than two players.
Unfortunately, the answer is negative.
The reason is that when $|I|>2$,
as mentioned earlier, the minmax value of a player can be strictly higher than her maxmin value. Consequently, when a player adopts a subgame-perfect $\delta$-maxmin strategy,
she is no longer guaranteed to obtain a payoff of at least her minmax value minus $\delta$. In particular, when all players follow subgame-perfect $\delta$-maxmin strategies, we cannot apply Corollary \ref{suff-corr} as in the proof above.
\end{remark}

\section{Second Proof: Regularity of the Minmax Value}

In this section we prove Theorem~\ref{theorem:main} using regularity of the minmax value.

\begin{definition}
Let $Y$ be a topological space,
and let $\mathcal{F}(Y)$ denote the Borel sigma-algebra of $Y$. 
A function $g :  \mathcal{F}(Y) \to \dR$ is \emph{inner-regular} if for each 
$Z\in \mathcal{F}(Y)$
we have
\[ g(Z)\, =\, \sup\{ g(C) \colon C \subseteq Z,\ C \hbox{ compact}\}. \]
\end{definition}

It is well known that every probability measure defined on the Borel sets of a compact metric space is inner-regular. In particular, for each strategy profile $\sigma$, the probability measure $\P_\sigma$ is inner-regular on the set $A^\dN$ of plays.

It follows from Martin (1998) that, given two nonempty and finite action sets $A_1$ and $A_2$, the function that assigns to each Borel-measurable set $W \in \mathcal{F}(A^\dN)$ the value of the two-player zero-sum Blackwell game $(\{1,2\}, (A_1, A_2), (\mathbf{1}_W,-\mathbf{1}_W)$ is inner-regular,
where $\mathbf{1}_W$ denotes the characteristic function of the set $W$.
Ashkenazi-Golan, Flesch, Predtetchinski, and Solan (2022b)
extended this result for the minmax value (and for the maxmin value) of multiplayer Blackwell games, and this is stated next. 

\begin{theorem}[Ashkenazi-Golan, Flesch, Predtetchinski, Solan (2022b)]
\label{theorem-reg}
Let $I$ and 
$(A_i)_{i \in I}$
be nonempty and finite sets. 
For each $i\in I$ and
$W_i \in\mathcal{F}(A^\dN)$, let $\overline{v}_i(W_i)$ denote the minmax value of player $i$ in the Blackwell game with $I$ as the set of players, $A_j$ as the action set of player $j$, for each $j\in I$, and $\mathbf{1}_{W_i}$ as the payoff function of player $i$. Then, 
\[ \overline{v}_i(W_i) \,=\, \sup\{\overline{v}_i(C_i) \colon C_i \subseteq W_i,\, C_i \hbox{ compact}\}. \]
\end{theorem}

We can now turn to the proof of Theorem~\ref{theorem:main} using regularity of the minmax value. The following proof is given in Ashkenazi-Golan, Flesch, Predtetchinski, and Solan (2022b).\medskip

\begin{proof}[Proof of Theorem~\ref{theorem:main} with regularity] In view of Lemma \ref{suff-cond}, it suffices to find a play $p^*$ such that $f_i(p^*)\geq \overline{v}_i-\tfrac{\ep}{2}$ for each player $i\in I$.

Let $\Gamma_\ep = (I,(A_i)_{i \in I},(\mathbf{1}_{W_i})_{i \in I})$ be the auxiliary Blackwell game, where 
\[ W_i \,:=\, \{ p \in A^\dN \colon f_i(p) \geq \overline{v}_i - \tfrac{\ep}{2}\},\qquad \forall i\in I. \]

For every player $i\in I$, we have $\overline{v}_i(W_i) = 1$. Indeed, consider a player $i\in I$, and a strategy profile $\sigma_{-i}$ of her opponents. It follows from Theorem \ref{Ayala} that, in the game $\Gamma$, player $i$ has a strategy $\sigma_i$ that is an $\tfrac{\ep}{2}$-best response 
to $\sigma_{-i}$ 
in each subgame. Hence, in particular, we have $\E_{\sigma_i,\sigma_{-i}}[f_i\mid h]\,\geq\,\overline{v}_i-\tfrac{\ep}{2}$ for each history $h$. By L\'{e}vy's zero-one law, we obtain $\P_{\sigma_i,\sigma_{-i}}(W_i)\,=1$. Consequently, $\overline{v}_i(W_i)=1$ as claimed.

Fix $\delta\in(0,\frac{1}{|I|})$. By Theorem \ref{theorem-reg}, for each player $i\in I$, there is a compact set $C_i \subseteq W_i$ such that $\overline{v}_i(C_i) \geq 1-\delta$. For every stage $t \in \dN$, let $D_i^t$ be the set of plays that agree with some play in $C_i$ in the first $t$ stages:
\[D^t_i\,=\,\{p\in A^\dN\colon p^{\leq t}=q^{\leq t}\text{ for some }q\in C_i\}.\]
Since $D^t_i \supseteq C_i$, we have $\overline{v}_i(D^t_i) \geq 1-\delta$.

Consider now the auxiliary Blackwell game $\Gamma_{\ep,t} = (I,(A_i)_{i \in I}, (\mathbf{1}_{D_i^t})_{i \in I})$, which is essentially a $t$-stage game, because the actions beyond stage $t$ do not influence the payoffs. As the game $\Gamma_{\ep,t}$ has finite horizon and the action sets of the players are finite, $\Gamma_{\ep,t}$ has a 0-equilibrium $\sigma^t = (\sigma_i^t)_{i \in I}$.
Then, for each player $i\in I$ we have 
\[\prob_{\sigma^t}(D^t_i) \,\geq\, \overline{v}_i(D^t_i) \,\geq\, 1-\delta.\]
Because $\delta<\frac{1}{|I|}$, we 
necessarily have
$\bigcap_{i \in I} D^t_i\neq\emptyset$.
Thus, for each $t\in\dN$, there is a play $p_{(t)} \in \cap_{i \in I} D^t_i$. Because $A^\dN$ is compact, the sequence $(p_{(t)})_{t=1}^\infty$ has an accumulation point~$p^*$, and as the sets $C_i$ are compact, 
$p^*\in\bigcap_{i \in I} C_i$. Since $C_i\subseteq W_i$ for each player $i\in I$, we have $f_i(p^*)\geq \overline{v}_i-\tfrac{\ep}{2}$ for each player $i\in I$, as desired.
\end{proof}

\section{Third Proof: Acceptable Strategy Profiles}

Solan (2018) defined the concept of \emph{acceptable} strategy profiles in the context of stochastic games.
In the context of repeated games,
these are strategy profiles that yield high payoff to all players in all subgames.

\begin{definition}
Consider a Blackwell game.
Let $\delta > 0$.
A strategy profile $\sigma^*$ is \emph{$\delta$-minmax acceptable} if
for every history $h\in H$ and every player $i \in I$,
\[ \E_{\sigma^*}[f_i\mid h] \,\geq\, \overline{v}_i(h) - \delta. \]
\end{definition}

In a two-player Blackwell game, the strategy profile $(\sigma^\delta_1,\sigma^\delta_2)$
where $\sigma^\delta_i$ is a subgame-perfect $\delta$-maxmin strategy of player~$i$, for each $i\in\{1,2\}$, is $\delta$-minmax acceptable (cf. Section \ref{sec-firstproof}). As the following theorem shows, $\delta$-minmax acceptable strategy profiles always exist even if there are more than two players.

\begin{theorem}[Ashkenazi-Golan, Flesch, Predtetchinski, Solan (2022a)]
\label{theorem:good:strategy}
In every Blackwell game, for every $\delta>0$,
there exists 
a $\delta$-minmax acceptable
strategy profile.
\end{theorem}

We provide a sketch of the proof of Theorem~\ref{theorem:good:strategy},
which relies on an extension of a deep result in Martin (1998) for two-player zero-sum Blackwell games.

For a function $d:H\to \dR$, a player $i\in I$, and a history $h\in H$, consider the following one-shot game $G_i^{d,h}$: Each player $j\in I$ chooses an action $a_j\in A_j$, simultaneously with the other players, which induces an action profile $a=(a_j)_{j\in I}$. Then, player $i$ receives the payoff $d(ha)$, where $ha$ is the extension of the history $h$ with the action profile $a$ at the next stage. Let $u^{d,h}_i(x)$ denote player $i$'s expected payoff in this one-shot game under a mixed action profile $x\in X$:
\[u^{d,h}_i(x)\,:=\,\sum_{a\in A}\left[\prod_{j\in I}x_j(a_j)\cdot d(ha)\right],\]
and let $\overline{\text{val}}^{d,h}_i$ denote player $i$'s minmax value in this one-shot game:
\[ \overline{\text{val}}^{d,h}_i \,:=\,\min_{x_{-i}\in X_{-i}}\max_{x_i\in X_i}\,u^{d,h}_i(x_i,x_{-i}).\]

The following result is an extension of Martin (1998).

\begin{theorem}[Ashkenazi-Golan, Flesch, Predtetchinski, Solan (2022a)]
\label{theorem-count}
Consider a Blackwell game.
Let $\eta >0$, and let $i\in I$ be a player. There exists a function $d:H\to\dR$ with the following properties:
\begin{enumerate}
    \item $ \overline{\text{val}}^{d,h}_i \,\geq\, \overline{v}_i(h)-\eta$ for every $h\in H$.
    \item For every strategy profile $\sigma$ that satisfies
\[u^{d,h}_i(\sigma(h))\,\geq\, \overline{\text{val}}^{d,h}_i \qquad\forall h\in H,\]
we have
\[\E_\sigma[f_i\mid h]\,\geq\, \overline{\text{val}}^{d,h}_i \qquad\forall h\in H.\]
\end{enumerate}
\end{theorem}
In words, the first property means that player $i$'s minmax value $ \overline{\text{val}}^{d,h}_i $ in the one-shot game $G_i^{d,h}$, at any history $h$, is at least her minmax value 
$\overline{v}_i(h)$
in the Blackwell game minus $\eta$. The second property says that if a strategy profile $\sigma$, at each history $h$, gives player $i$ a payoff of at least her minmax value $ \overline{\text{val}}^{d,h}_i $ in the one-shot game $G_i^{d,h}$, then this strategy profile $\sigma$, in the subgame of the Blackwell game at any history $h$, gives player $i$ a payoff of at least $\overline{\text{val}}^{d,h}_i$.

\bigskip

\begin{proof}[Proof of Theorem~\ref{theorem:good:strategy}]
Let $\eta\in(0,\delta)$, and for each player $i\in I$ let $d_i:H\to\dR$ be as in Theorem \ref{theorem-count}.

At each history $h$ consider the following one-shot game: Each player $i\in I$ chooses an action $a_i\in A_i$, simultaneously with the other players, which induces an action profile $a=(a_i)_{i\in I}$. Then, each player $i\in I$ receives the payoff $d_i(ha)$. Since the action sets are finite, this game has an equilibrium $x^h\in X$. Note that $x^h$ gives each player $i\in I$ an expected payoff of at least her minmax value in this one-shot game, i.e., $u_i^{d_i,h}(x^h)\geq
\overline{\text{val}}_i^{d_i,h}$ for each player $i\in I$.

Define the strategy profile $\sigma^*$ by letting $\sigma^*(h)=x^h$ for each history $h$. Then, by Theorem \ref{theorem-count}, for each player $i\in I$ 
we find
\begin{equation}
\label{atleastdelta}
\E_{\sigma^*}[f_i\mid h]\,\geq\,
\overline{\text{val}}^{d,h}_i
\, \geq\, \overline{v}_i(h)-\eta\,\geq\, \overline{v}_i(h)-\delta\qquad\forall h\in H,
\end{equation}
which completes the proof.
\end{proof}\bigskip

\begin{proof}[Proof of Theorem~\ref{theorem:main} with acceptable strategy profiles]
By Theorem~\ref{theorem:good:strategy}, the game $\Gamma$ admits an $\frac{\ep}{2}$-minmax acceptable strategy profile. As the payoff functions of $\Gamma$ are tail-measurable,
Lemma~\ref{indepvalue} implies that $\overline{v}_i=\overline{v}_i(h)$ for each player $i\in I$ and each history $h\in H$. 
Hence, Corollary~\ref{suff-corr} implies now that $\Gamma$ has an $\ep$-equilibrium.
\end{proof}

\begin{remark}\label{count-remark}\rm 
The proof method above can be slightly adjusted to show that Theorem~\ref{theorem:main} holds even if the number of players is countably infinite, see Ashkenazi-Golan, Flesch, Predtetchinski, Solan (2022a). Since this involves more technical details on the measure-theoretic foundation of the model and the results, we only briefly discuss how the method can be applied. 

Let $I=\dN$, and let $d_i$ be a function as in the proof above for each player $i\in I$. Fix for each player $i\in I$ an arbitrary action $\widehat a_i\in A_i$.

Consider a strategy profile that prescribes to play at each stage $t\in\dN$ as follows: Each player $k>t$ plays the action $\widehat a_k$. Given these actions, the remaining finitely many players, i.e, players $1,\ldots,t$, play an equilibrium in the one-shot game given by the functions $d_1,\ldots,d_t$.

Thus, each player $i$ plays the action $\widehat a_i$ at stages $1,\ldots,i-1$, and plays according to the function $d_i$ at stages $i,i+1,\ldots$. Since the payoffs are tail-measurable, the first $i-1$ stages do not matter for player $i$'s payoff, and therefore playing according to $d_i$ at all stages from stage $i$ on guarantees that player $i$'s expected payoff is at least $\overline{v}_i-\delta$, similarly to \Eqref{atleastdelta}. 
\end{remark}

\section{Fourth Proof: Detection of Deviations}

All the $\ep$-equilibria we constructed so far had a similar structure:
the players follow some fixed play that yields all players a payoff at least their minmax value minus $\ep$,
and a deviation triggers an indefinite punishment.
In this section we construct an $\ep$-equilibrium that involves mixtures on the equilibrium path.

Alon, Gunby, He, Shmaya, and Solan (2022)
studied a model where there is a certain \emph{goal},
defined by a set $D \subseteq A^\dN$,
and the players are supposed to follow a given strategy profile $\sigma^*$ such that $\prob_{\sigma^*}(D)$ is close to 1.
Suppose that the realized play turns out to be outside $D$.
Assuming that at most one player deviated,
the question is whether an outside observer who observes the play can correctly identify the deviator 
with high probability.
As the following result states, the answer is positive.

\begin{theorem}[Alon, Gunby, He, Shmaya, and Solan (2022)]
\label{theorem:blame}
Consider a Blackwell game. 
Let $\sigma^* \in \Sigma$ be a strategy profile, let $\delta > 0$, and let $D \subseteq A^\dN$ satisfy 
$\prob_{\sigma^*}(D) > 1-\delta$.
There exists a function $g : D^c \to I$ such that for every $i \in I$ and every $\sigma_i \in \Sigma_i$, 
\[ \prob_{\sigma_i,\sigma^*_{-i}}(D^c \hbox{ and } \{g \neq i\})\, <\, 2\sqrt{(|I|-1)\delta}. \]
\end{theorem}


\bigskip

\begin{proof}[Proof of Theorem~\ref{theorem:main} with detection of deviations]
Fix $\delta > 0$ such that 
\begin{equation}
\label{choosing-delta}
\ep\,>\,4\delta + 4\cdot \left(\sqrt{(|I|-1)\cdot 2\delta}+\delta\right)\cdot\max_{i\in I}||f_i||_{\infty}.
\end{equation}
By Theorem~\ref{theorem:good:strategy},
there is a $\delta$-minmax acceptable strategy profile $\widehat \sigma$, that is, 
\[ \E_{\widehat \sigma}[f_i\mid h] \,\geq\, \overline{v}_i(h) - \delta \,=\, \overline{v}_i - \delta,\qquad\forall i \in I,\ \forall h \in H. \]
Since the payoff functions are bounded, there is
$c \in \dR^I$ such that 
$c_i \geq \overline{v}_i - \delta$ for every $i \in I$ and
\color{black}
\[ \prob_{\widehat \sigma}(\|f - c\|_\infty < \delta) > 0, \]
where $f=(f_i)_{i\in I}$. Since $f$ is Borel-measurable,
there is a history 
$h^* \in A^{t^*}$, for some stage $t^*\in\dN$,
such that 
\begin{equation*}
\prob_{h^*,\widehat \sigma}(\|f - c\|_\infty < \delta) > 1-\delta,
\end{equation*} 
where $\prob_{h^*,\widehat \sigma}$ is the conditional probability induced by $\widehat\sigma$ on the event that $h^*$ occurred.
Formally, it is the probability distribution over $A^\dN$ defined by the strategy profile that follows $h^*$ in the first $t^*$ stages, and then follows $\widehat\sigma$.
\black
Denote 
\[ D^* \,:=\, \left\{ p\in A^{\dN} \,\colon \|f(p) - c\|_\infty < \delta \right\}. \]
Since every probability measure on $A^\dN$ is inner-regular, 
there is a compact subset $D \subseteq D^*$ such that 
\begin{equation}
    \label{equ:d}
\prob_{h^*,\widehat \sigma}(D) \,\geq\, \prob_{h^*,\widehat \sigma}(D^*) - \delta \,>\, 1-2\delta. 
\end{equation} 
Since $D$ is compact, its complement $D^c$ is open, 
hence it is the union of basic open sets:
there is a minimal subset $H^* \subseteq H$ such that 
\[ D^c \,=\, \bigcup\,\{ C(h) \,\colon\, h \in H^*\}, \]
where $C(h)$ 
is the set of plays that extend $h$.
By Theorem~\ref{theorem:blame}, there exists a function $g : D^c \to I$ such that for every $i \in I$ and every $\sigma_i \in \Sigma_i$, 
\[ \prob_{h^*,\sigma_i,\widehat \sigma_{-i}}(D^c \hbox{ and } \{g \neq i\}) \,<\,\widehat\ep\, :=\,
2\sqrt{(|I|-1)\cdot 2 \delta}. \]

Consider the following strategy profile $\sigma^*$:
\begin{itemize}
\item 
In the first 
$t^*$
stages, 
the players follow $h^*$.
Any deviation from this history triggers an indefinite punishment at the minmax level; that is, if player~$i$ deviated from $h^*$, then the other players switch to the strategy profile $\sigma_{-i}^{h,i}$ (cf. \Eqref{pun}), which ensures that player~$i$'s payoff is at most $\overline v_i + \delta$,
where $h$ is the minimal history that is not a prefix of $h^*$.
\item
From stage $t^*+1$
and on, the players follow $\widehat \sigma_{h^*}$, 
i.e., the continuation of the strategy profile $\widehat \sigma$ after $h^*$.
\item
If the realized 
history at some stage
coincides with some history $h \in H^*$,
the players switch to a punishment strategy against player $g(h)$; that is, the players $I \setminus \{g(h)\}$ switch to the strategy profile $\sigma_{-g(h)}^{h,g(h)}$, 
which ensures that player~$g(h)$'s payoff is at most $\overline v_{g(h)} + \delta$. 
\end{itemize}

We argue that the strategy profile $\sigma^*$ is an $\ep$-equilibrium.
Indeed, 
by Eq.~\eqref{equ:d},
under $\sigma^*$,
the realized history coincides with some history in $H^*$ with probability at most 
$2\delta$.
Since the payoff at all plays in $D$ is $\delta$-close to $c$,
we obtain for each player $i\in I$ that 
\[ \E_{\sigma^*}[f_i] \,\geq\, c_i - 2 \delta\cdot(1 + 2\|f_i\|_\infty). \]
Fix now a player $i \in I$ and a strategy $\sigma_i \in \Sigma_i$.
Denote by $\mu$ the probability that under $(\sigma_i,\sigma^*_{-i})$ the realized history at some stage lies in $H^*$.
Since once the realized history $h$ lies in $H^*$ player $g(h)$ is punished,
since the probability that in this case 
$g(h) \neq i$ is at most $\widehat\ep$, 
and since $c_i \geq \overline v_i - \delta$, 
we have 
\begin{eqnarray*}
\E_{\sigma_i,\sigma^*}[f_i] &\leq& (1-\mu)(c_i + \delta) + \widehat\ep \|f_i\|_\infty + \max\{0,\mu-\widehat\ep\}\cdot (\overline{v}_i + \delta)\\
&\leq& c_i + 2\delta + 2\widehat\ep \|f_i\|_\infty.
\end{eqnarray*}
Hence, by \Eqref{choosing-delta} we have $\E_{\sigma_i,\sigma^*}[f_i]-\E_{\sigma^*}[f_i]\leq\ep$, which proves $\sigma^*$ is an $\ep$-equilibrium.
\end{proof}

\section{Concluding remarks}
We studied the existence of $\ep$-equilibrium in Blackwell games under the condition that the payoffs are tail-measurable. When the number of players is finite, this condition was only used to obtain that the minmax values of the players are independent of the history (cf. Lemma~\ref{indepvalue}). For countably infinitely many players, however, our proof technique  did require tail-measurability in its full generality (cf. Remark~\ref{count-remark}).

There are many interesting directions for further research. For example, to find conditions under which a Blackwell game with tail-measurable payoffs admits a 0-equilibrium, or to investigate the existence of $\ep$-equilibrium in Blackwell games without requiring that the payoffs are tail-measurable. A challenging but rather open direction is to investigate Blackwell games when there is an underlying state space (cf. Remark~\ref{rem-stoch}).

We hope that our paper will attract researchers to the study of the fascinating area of Blackwell games.

\end{document}